\documentclass[12pt,amstex]{amsart}

\usepackage{mathptmx}
\usepackage{mathrsfs}

\usepackage{verbatim}
\usepackage{url}
\usepackage[all]{xy}
\usepackage{color}

\usepackage[colorlinks=true,citecolor=blue]{hyperref}

\usepackage{stmaryrd}
\usepackage{epsfig}
\usepackage{amsmath}
\usepackage{amssymb}

\usepackage{amscd}
\usepackage{graphicx}
\usepackage{pstricks}

\theoremstyle{plain}
\newtheorem{thm}{Theorem}[section]
\newtheorem{lem}[thm]{Lemma}

\newtheorem{ques}[thm]{Question}
\newtheorem{cor}[thm]{Corollary}
\theoremstyle{definition}
\newtheorem{de}[thm]{Definition}

\theoremstyle{remark}
\newtheorem{rem}[thm]{Remark}

\numberwithin{equation}{section}

\def \N {\mathbb N}
\def \C {\mathbb C}
\def \Z {\mathbb Z}

\def \X {\mathcal{X}}

\def \d {\delta}

\def \lra{\longrightarrow}

\begin{document}
\title[Multiple Ergodic averages]{almost sure convergence of the multiple ergodic average for certain weakly mixing  systems}

\author{Yonatan Gutman}

\address{Institute of Mathematics, Polish Academy of Science, ul. \'{S}niadeckich 8, 00-656 Warszawa, Poland}
\email{y.gutman@impan.pl}

\author{Wen Huang}
\author{Song Shao}
\author{Xiangdong Ye}

\address{Wu Wen-Tsun Key Laboratory of Mathematics, USTC, Chinese Academy of Sciences and
Department of Mathematics, University of Science and Technology of China,
Hefei, Anhui, 230026, P.R. China.}

\email{wenh@mail.ustc.edu.cn}\email{songshao@ustc.edu.cn}
\email{yexd@ustc.edu.cn}



\thanks{Y. Gutman was partially supported by the the National
Science Center (Poland) grant 2013/08/A/ST1/00275 and by the the National
Science Center (Poland) grant 2016/22/E/ST1/00448. W. Huang, S. Shao and X. Ye are supported by NNSF of China (11371339, 11431012, 11571335, 11225105) and by ``the Fundamental Research Funds for the Central Universities''.
}


\begin{abstract}
The family of pairwise independently determined (PID) systems, i.e. those for which the independent joining is the
only self joining  with independent 2-marginals, is a class of systems for which the long standing open question
by Rokhlin, of whether mixing implies mixing of all orders, has a positive answer.
We show that in the class of weakly mixing PID one finds a positive answer for another long-standing open problem, whether the  multiple
ergodic averages
\begin{equation*}
    \frac 1 N\sum_{n=0}^{N-1}f_1(T^nx)\cdots
f_d(T^{dn}x), \quad N\to \infty,
\end{equation*}
almost surely converge.
\end{abstract}

\maketitle

\section{Introduction}\label{}

Ergodic theory is the study of the qualitative properties of measure preserving transformations. A
quadruple $(X,\X,\mu,T)$ is a measure preserving transformation (m.p.t. for short) if $(X,\X,\mu)$
is a measurable space with $\mu(X)=1$, and $T : X \rightarrow X$ is a m.p.t. That is, for $A\in \X$, $T^{-1}A\in \X$ and $\mu(T^{-1}A)=\mu(A)$.
In this paper, we assume that $T$ is invertible and both $T$ and $T^{-1}$ are m.p.t.
We will use $Tf$ to denote the function $f(Tx)$, i.e. we treat $T$ as a unitary operator.

\medskip

The family of pairwise independently determined (PID) systems, i.e. those for which the independent
joining is the only self-joining  with independent 2-marginals, is a class of systems for which the
long standing open question by Rokhlin, of whether mixing implies mixing of all orders, has a positive
answer. Our goal in this paper is to show that in the class of weakly mixing PID one finds a positive answer for
another long-standing open problem, whether the  multiple
ergodic averages
\begin{equation*}
    \frac 1 N\sum_{n=0}^{N-1}f_1(T^nx)\cdots
f_d(T^{dn}x), \quad N\to \infty,
\end{equation*}
almost surely converge.

\medskip

First let us recall some
results related to the convergence of ergodic averages.
The first pointwise ergodic theorem was proved by Birkhoff in 1931 (\cite{Bir31}).
Following  Furstenberg's beautiful work on the dynamical proof of Szemeradi's theorem in 1977 \cite{F77}, problems concerning the convergence of multiple ergodic averages (in $L^2$ norm or pointwisely) started attracting a lot of attention in the literature.

\medskip

The convergence of the averages
\begin{equation}\label{MEA-one}
    \frac 1 N\sum_{n=0}^{N-1}f_1(T^nx)\ldots
f_d(T^{dn}x)
\end{equation}
in $L^2$ norm was established by Host and Kra \cite{HK05} (see also Ziegler \cite{Z}). The convergence of the multiple ergodic average for commuting
transformations was obtained by Tao \cite{Tao} using the finitary ergodic method, see \cite{Austin,H2009} for
more traditional ergodic proofs by Austin and Host respectively. There is also a proof by Towsner using non-standard analysis (\cite{Tow09}). The convergence of multiple ergodic averages for nilpotent group actions
was proved by Walsh \cite{Walsh}.

\medskip

The first breakthrough on pointwise convergence of (\ref{MEA-one}) for $d > 1$ is due to Bourgain,
who showed in \cite{B90} that for $d = 2$, for $p,q\in \N$ and for
all $f_1, f_2 \in  L^\infty$,
the limit of $\frac{1}{N}\sum_{n=0}^{N-1} f_1(T^{np}x)f_2(T^{nq}x)$ exists a.s.
In \cite{DL1996}, Derrien and Lesigne showed
that the problem of the almost sure convergence of the multiple ergodic averages can be reduced
to the case when the m.p.t. has zero entropy. To be precise,
let $(X,\X,\mu,T)$ be a m.p.t. Let $d\in \N$, and $p_1(n)$, $\ldots$, $p_d(n)$ $\in \Z[n]$. The limit of the multiple ergodic average
\begin{equation*}
        \frac{1}{N} \sum_{n=0}^{N-1} f_1(T^{p_1(n)}x)\ldots f_d(T^{p_d(n)}x)
\end{equation*}
exists a.s. for all $f_1,\ldots, f_d$ in $L^\infty(X,\X,\mu)$ if and only if it exists a.s. for all $f_1,\ldots,f_d$
in $L^\infty(X,\mathcal{P},\mu)$, where $\mathcal{P}$ is the Pinsker factor. In particular, the almost sure convergence of this average holds for K-systems.
Recall that a m.p.t. is a K-system if its Pinsker factor is trivial.

Recently, Huang, Shao and Ye \cite{HSY} showed the a.s. convergence of (\ref{MEA-one}) for distal systems. That is, let $(X,\X,\mu, T)$ be an ergodic measurable distal system,
and $d\in \N$. Then for all $f_1, \ldots, f_d \in L^{\infty}(\mu)$ the averages
\begin{equation*}
    \frac 1 N\sum_{n=0}^{N-1}f_1(T^nx)\ldots f_d(T^{dn}x)
\end{equation*}
converge $\mu$ a.s.
Note that the Furstenberg-Zimmer structure theorem \cite{F,z76a,z76b} states that each ergodic system is a weakly mixing extension of
an ergodic measurable distal system. Thus, by the above theorem the question on the a.s. convergence of (\ref{MEA-one})
can be reduced to the question of how to lift the convergence though weakly mixing extensions. We note that
in \cite{DS15, DS16} Donoso and Sun generalized the above result to commuting distal transformations.

\medskip

Even for weakly mixing systems, the question on the a.s. convergence of (\ref{MEA-one}) still remains open.
A partial answer to this question was obtained by Assani \cite{Assani1998}, who showed that if $(X,\X,\mu,T)$ is a
weakly mixing system such that the restriction of $X$ to its Pinsker algebra has spectral type singular w.r.t Lebesgue measure, then the limit of $(\ref{MEA-one})$ exists a.s.

\medskip

Our main result in this paper can be stated as follows.

\medskip
\noindent {\bf Main Theorem:} {\it
Let $(X,\mathcal{X},\mu,T)$ be a weakly mixing and pairwise independently
determined (PID) m.p.t. Then for all $d\in\N$ and all $f_1,\ldots,f_d\in L^\infty(X,\X,\mu)$,}
\begin{equation*}
    \frac 1 N\sum_{n=0}^{N-1}f_1(T^nx)\ldots
f_d(T^{dn}x)\stackrel{a.s.}{\longrightarrow } \int f_1d\mu\int f_2d\mu \ldots \int f_dd\mu , \quad N\to \infty.
\end{equation*}

We also observe that for a generic  m.p.t. $(\ref{MEA-one})$ exists a.s.

\medskip

The paper is organized as follows. In the next section we will discuss Rokhlin's multifold mixing question and pairwise independent
joinings, and explain the connection of this problem with our main theorem. In the final section  we will present the proof of the main
theorem and its corollaries.
For example, by our theorem, one can show that the a.s. convergence of (\ref{MEA-one}) for finite-rank  mixing m.p.t., simple systems etc.

\medskip

There are excellent surveys on multiple recurrence and multiple ergodic averages,  see for
example \cite{Bergelson06, F88, F90, F10, Kra}. And for other progress on the a.s. convergence of
multiple ergodic averages, we refer to \cite{Hou, A2, CF}.

\section{Rokhlin's multifold mixing question and pairwise-independent joinings}\label{sec-4}
Somewhat surprisingly the problem of almost sure convergence is related to another well-known and long-standing question in ergodic theory. We first recall this problem and then discuss the connection.

\subsection{Rokhlin's Multifold Mixing Question}

Rokhlin defined {\em multifold mixing} in \cite{Rokhlin} as follows: a m.p.t. $T$ is said to be {\em $k$-fold mixing}
if for all $A_0$, $A_1,\ldots$, $A_k\in \X$,
\begin{equation*}
\lim_{n_1,\ldots,n_k\to\infty} \mu(A_0\cap T^{-n_1}A_1\cap T^{-(n_1+n_2)}A_2\cap \ldots \cap T^{-(n_1+\ldots +n_k)}A_k)=\prod_{i=0}^k\mu(A_i).
\end{equation*}

Of course, $k$-fold mixing implies $j$-fold mixing if $k\ge j$. Rokhlin asked in his article whether the converse is true.
\begin{ques}\label{Rokhlin-conjecture}
Does mixing imply mixing of all orders?
\end{ques}

This is one of the outstanding open questions in ergodic theory. Kalikow \cite{Kal84} proved that Rokhlin's problem is true for rank one systems.
Host \cite{H91} showed that Rokhlin's problem is true for systems with spectral type singular w.r.t Lebesgue measure.
Kalikow's result was extended by Ryzhikov \cite{Ryzhikove} to finite rank systems.
For more notable advances on this question, see \cite{Mar78, Kal84, H91, Ryzhikove, Sta93, Tik07, Bas13,FK14}. And we refer
to \cite{DK78, GM78, Led78} for related counterexamples.

\subsection{Pairwise-Independent Joinings}

The notion of joinings was introduced by Furstenberg \cite{F67}.
Given an integer $d\ge 2$, a {\em joining} of $d$ systems $(X_i,\X_i,\mu_i,T_i), 1\le i\le d$ is a
probability measure $\lambda$ on the product space $\prod_{i=1}^d (X_i,\X_i)$ which is invariant under the
transformation $T_1\times \ldots \times T_d$ and whose marginal projection on each $X_i$ is equal to $\mu_i$. When $(X_1,\X_1,\mu_1,T_1)=\ldots=(X_d,\X_d,\mu_d,T_d)$, we then say that
$\lambda$ is a {\em $d$-fold self-joining}.

\medskip

The joining $\lambda$ is {\em pairwise independent} if its projection on $X_i\times X_j$ is
equal to $\mu_i\times \mu_j$ for all $i\neq j\in \{1,2,\ldots,d\}$, and it is {\em independent} if it is the product measure. A system $(X,\mathcal{X},\mu,T)$ is said to be {\em pairwise independently
determined} (PID) if all pairwise independent $d$-self joinings $(d\geq 3)$ are independent.

\medskip
\subsection{The Relation between Rokhlin's Question and the a.s Convergence Question}

It is well known that a negative answer to the following question would solve Rokhlin's problem (see \cite[Section 10.8]{Nad98} as well as \cite[Proposition 3.2]{dlR06}):

\begin{ques}\label{prob: invertible zero-entropy weakly mixing is PID}\cite[p. 552]{DJR87},\cite[Question 14]{DLR09}
Let $(X,\mathcal{X},\mu,T)$ be a zero-entropy, weakly mixing m.p.s. Is it PID?
\end{ques}

As mentioned above the problem of the almost sure convergence of the multiple ergodic averages may be reduced to the case when the m.p.t. has zero entropy. Thus according to our Main Theorem (see above or Theorem \ref{thm:PID implies a.e}) an affirmative answer to Question \ref{prob: invertible zero-entropy weakly mixing is PID} will prove the almost sure convergence of multiple ergodic averages for weakly mixing systems. We would like however to stress that  we are not familiar with a direct method which relates the two questions.

\subsection{Classes of PID systems}

Question \ref{prob: invertible zero-entropy weakly mixing is PID} was solved by Host in the affirmative for systems with spectral type singular w.r.t Lebesgue measure and by Ryzhikov for finite rank systems (see below).
However it is still open for the general case. We also remark that no counter-example is known even
removing the weak mixing assumption.

\begin{thm}[\cite{H91}]{\rm (Host's Theorem on systems with spectral type singular w.r.t Lebesgue measure\footnote{These systems are also known as {\em systems having purely singular spectrum}.})}\label{Host-thm}
Let $(X_i,\X_i,\mu_i,T_i)$, $i=1,2,\ldots, d$ be m.p.t, at least $d-2$ of which are weakly mixing with
spectral type singular w.r.t Lebesgue measure.
Then every pairwise independent joining of $T_1,\ldots, T_d$ is independent.
\end{thm}

Note that we say the spectral type of a m.p.t is of singular if it is singular with respect to Lebesgue measure on $\mathbb{T}$.

\begin{thm}[\cite{Ryzhikove}]{\rm (Ryzhikov's Theorem for Finite Rank Systems)}\label{Ryzhikov-thm}
Let $(X,\X,\mu,T)$ be a finite-rank mixing transformation then it is PID.
\end{thm}

Theorem \ref{Host-thm} and Theorem \ref{Ryzhikov-thm} are two important results on Question \ref{prob: invertible zero-entropy weakly mixing is PID}.
In \cite[Definition 7]{Tho95}
Thouvenot, following Ratner, introduced the $R_{p}$\footnote{Also referred to as the  R-property (index $p$ is implicit).}
($p\neq0)$ property for certain continuous $\mathbb{{R}}-$flows
$\{T_{t}\}_{t\in\mathbb{{R}}}$. In her groundbreaking work Ratner
showed that the classical horocycle flows on the unit tangent bundle
of a surface of constant negative curvature with finite volume have
$R_{p}$ for all $p\neq0$ \cite{Rat82a,Rat82b,Rat83}. One can show that
any discretization of the flow $\{T_{nt_{0}}\}_{n\in\mathbb{{Z}}}$
which is ergodic and $R_{t_{0}}$ is PID ( \cite[p. 8569]{Lem09}). Notice that $R_{t_{0}}$ is ergodic for all $t_{0}\in \mathbb{R}$ except possibly for a countable set (\cite[Lemma 12.1]{CFS}). Thus the classical horocycle flows furnish examples of PID flows. One can prove similar theorems
for weakenings of the $R$-property (\cite{FK14}).

\begin{lem}
\label{thm:A-weakly-mixing isometric of PID is PID}A weakly mixing\footnote{This refers to $Y$ and should not be confused with a relatively weakly
mixing extension.} isometric extension $Y=X\times_{\sigma}K/H$ of a PID action is again PID.
\end{lem}

\begin{proof}
As in \cite[Lemma 5.2]{DJR87} which is stated for group extensions
however the proof works also for isometric extensions as uniqueness
of Haar measure holds in this case too (\cite[Theorem 2.3.5]{AM07}).
\end{proof}

\begin{rem}
\cite[Theorem 3]{Rob92} For an arbitrary weakly mixing $X$ and a compact
group $K$ it is a generic property for cocycles $\sigma$ that $X\times_{\sigma}K$
is again weakly mixing.
\end{rem}

\subsection{JPID and PID}

\begin{de}(\cite[p. 449]{DJR87})
A family of m.p.t $\{(X_{i},\mathcal{X}_{i},\mu_{i},T_{i})\}_{i=1}^n$
is said to be {\em jointly pairwise
independently determined} (JPID) if any joining on $X_{1}\times X_{2}\times\cdots X_{n}$
which is pairwise independent must be independent.

Thus $(X,\mathcal{X},\mu,T)$ is PID if for all $n\in\N$, any $n$ copies of $X$ are JPID.
\end{de}

The notions of PID and JPID are connected by the following theorem.
\begin{thm}(\cite[Proposition 5.3]{DJR87})\label{thm:PIDs are JPID}
Let $\{(X_{i},\mathcal{X}_{i},\mu_{i},T_{i})\}_{i=1}^n$
be  m.p.t. If each of them is PID, then they are JPID.
\end{thm}

The following lemma will be used in the next section. We note that
if $\mu$ and $\nu$ are two invariant measures, $\nu\ll \mu$ and $\mu$ is ergodic then $\nu=\mu$ \cite[Remarks of Theorem 6.10]{Walters}.

\begin{lem}\label{lem:PIDn is PID}
Let $(X,\mathcal{X},\mu,T)$ be a weakly mixing PID m.p.t. Then $(X,\mathcal{X},\mu,T^{n})$
is PID for any $n\in\Z$ with $n\not=0$.
\end{lem}

\begin{proof} Since $(X,\mathcal{X},\mu,T)$ is PID if and only if $(X,\mathcal{X},\mu, T^{-1})$ is PID, we only need to consider the case when $n\ge 2$.
Let $\lambda$ be a $T^{n}-$joining on $X^{k}$ which is pairwise
independent. Define
$$T^{(k)}=T\times T\times  \cdots \times T \, \text{ ($k$-times) and }
\rho=\frac{1}{n}\sum_{i=0}^{n-1}(T^{(k)}_{*})^{i}\lambda.$$
As $(T^{(k)}_{*})^{n}\lambda=\lambda$, $(T^{(k)}_{*})\rho=\rho$. Denote by $\pi_{lr}:X^{k}\rightarrow X$
the projection on the $l$-th and $r$-th coordinates. We claim $\rho$
is pairwise independent. Indeed, $$\pi_{lr}\rho=\pi_{lr}(\frac{1}{n}\sum_{i=0}^{n-1}(T^{(k)}_{*})^{i}\lambda)=\frac{1}{n}
\sum_{i=0}^{n-1}\pi_{lr}((T^{(k)}_{*})^{i}\lambda)=\mu\times\mu,$$
as
\begin{eqnarray*}
(T^{(k)}_{*})^i\lambda \big( \pi_{lr}^{-1}(A) \big)&=&\lambda\big( (T^{(k)})^{-i}\pi_{lr}^{-1}(A)\big)\\
&=&\lambda \big( \pi_{lr}^{-1}((T\times T)^{-i}A)\big)\\ &=& \mu\times\mu((T\times T)^{-i}A)= \mu\times\mu(A)
\end{eqnarray*}
for $A$ measurable in $\mathcal{X}\times \mathcal{X}$. As $(X,\mathcal{B},\mu,T)$ is PID,
we conclude $\rho=\mu^{k}$. As $(X,\mathcal{B},\mu,T)$ is weakly
mixing, $\mu^{k}$ is $(T^{(k)})^{n}$-ergodic. Combining this with the fact that $\mu^{k}=\rho=\frac{1}{n}\Sigma_{i=0}^{n-1}(T^{(k)}_{*})^{i}\lambda$ and each $(T^{(k)}_{*})^{i}\lambda$ is $(T^{(k)})^{n}$-invariant for $i=0,1,\cdots,k-1$, one has $\lambda=\mu^{k}$.
\end{proof}





\section{Multiple ergodic averages for weakly mixing systems}\label{sec-5}

 In this section we use the idea of models to prove pointwise convergence of multiple ergodic averages for a weakly mixing PID m.p.t.
 We give some applications, particularly a simpler proof for Assani's result \cite{Assani1998}.  We start with a simple observation.

\begin{lem}\label{lem-unique}
Let $(X,T)$ be a uniquely ergodic weakly mixing system and $n\in \N$. Then $(X,T^n)$ is also uniquely ergodic.
\end{lem}

\begin{proof}
Let $\mu$ be the unique invariant measure of $(X,T)$ and $\nu$ be any $T^n$-invariant measure. Put
$\mu'=\frac 1n (\nu+T\nu +\ldots +T^{n-1}\nu).$
Then $T\mu'=\mu'$ which implies that $\mu'=\mu$. That is,
$\mu=\frac 1n (\nu+T\nu +\ldots +T^{n-1}\nu).$ As $\mu$ is also $T^n$-ergodic, we conclude that $\mu=\nu=T\nu=\ldots =T^{n-1}\nu$
which proves the lemma.
\end{proof}

Let $(X,T)$ be a topological dynamical system . We denote by $M(X)$ the collection of all probability Borel measures on $X$.
\begin{lem}\label{lem-bougain}
Let $(X,\X,\mu,T)$ be a weakly mixing m.p.t. and $(X,T)$ be uniquely ergodic. Then for all $i\neq j \in \N$,
there is some $X_{i,j}\in \X$ such that $\mu(X_{i,j})=1$ and for all $x\in X_{i,j}$
\begin{equation*}
  \frac{1}{N} \sum_{n=0}^{N-1} (T^i\times T^j)^n\d_{(x,x)}\longrightarrow
\mu\times \mu, \ N\to \infty, \quad \text{weakly in $M(X\times X)$.}
\end{equation*}
\end{lem}

\begin{proof}
By Bourgain's double recurrence theorem \cite{B90}, for all $f,g\in C(X)$, the limit of $\displaystyle \frac 1 N\sum_{n=0}^{N-1}f(T^{in}x)g(T^{jn}x)$ exists a.s. Since $(X,\X,\mu,T)$ is weakly mixing, $$\displaystyle  \lim_{N\to\infty }\frac 1 N\sum_{n=0}^{N-1}f(T^{in}x)g(T^{jn}x)\stackrel{L^2}{=} \int f d\mu\int g d\mu.$$ Hence one has that for all $f,g\in C(X)$,
\begin{equation}\label{B-2}
  \lim_{N\to\infty}\frac 1 N\sum_{n=0}^{N-1}f(T^{in}x)g(T^{jn}x)\stackrel{a.s.}{= } \int f d\mu\int g d\mu .
\end{equation}

Let $\{f_k\}_{k=1}^\infty$ be a countable dense subset of $C(X)$. According to (\ref{B-2}), for any pair $(k_1,k_2)$, there is $X^{(k_1,k_2)}\in \X$ such that $\mu(X^{(k_1,k_2)})=1$ and for all $x\in X^{(k_1,k_2)}$
\begin{equation*}
  \lim_{N\to\infty}\frac 1 N\sum_{n=0}^{N-1}f_{k_1}(T^{in}x)f_{k_2}(T^{jn}x){= } \int f_{k_1} d\mu\int f_{k_2} d\mu  .
\end{equation*}
Put $X_{i,j}=\bigcap_{k_1, k_2=1}^\infty X^{(k_1,k_2)}$. We have that $\mu(X_{i,j})=1$ and for each $x\in X_{i,j}$,
\begin{equation*}
  \lim_{N\to\infty}\frac 1 N\sum_{n=0}^{N-1}f_{k_1}(T^{in}x)f_{k_2}(T^{jn}x){=} \int f_{k_1} d\mu\int f_{k_2} d\mu
\end{equation*}
holds for all $k_1,k_2\in \N$. By approximating given continuous functions by members of $\{f_k\}_{k=1}^\infty$, one has
that for each $x\in X_{i,j}$,
\begin{equation*}
\begin{split}
\lim_{N\to\infty} \frac 1 N\sum_{n=0}^{N-1} (T^i\times T^j)^n\d_{(x,x)}\big(f\otimes g\big)&= \lim_{N\to\infty} \frac 1 N\sum_{n=0}^{N-1}f(T^{in}x)g(T^{jn}x)\\ &=  \int f d\mu\int g d\mu=\mu\times \mu\big(f\otimes g\big)
\end{split}
\end{equation*}
holds for all $f,g\in C(X)$.
The proof is completed.
\end{proof}

\begin{lem}\label{lem-product}
Let $\{a_i\}, \{b_i\}\subseteq \C$. Then
\begin{equation}\label{}
\prod_{i=1}^k a_i-\prod_{i=1}^k b_i=(a_1-b_1)b_2\ldots b_k+
a_1(a_2-b_2)b_3\ldots b_k +a_1\ldots a_{k-1}(a_k-b_k).
\end{equation}
\end{lem}

Thus, if $|a_i|,|b_i|\le 1$ for all $1\le i\le k$ then $|\prod_{i=1}^k a_i-\prod_{i=1}^k b_i|\le \sum_{i=1}^k|a_i-b_i|$.
Now we are ready to show the main result.

\begin{thm}\label{thm:PID implies a.e}
Let $(X,\mathcal{X},\mu,T)$ be a weakly mixing PID m.p.t. Then for all $d$ and all $f_1,\ldots,f_d\in L^\infty(X,\X,\mu)$,
\begin{equation*}
    \frac 1 N\sum_{n=0}^{N-1}f_1(T^nx)\ldots
f_d(T^{dn}x)\stackrel{a.s.}{\longrightarrow } \int f_1d\mu\int f_2d\mu \ldots \int f_dd\mu , \quad N\to \infty.
\end{equation*}
\end{thm}
\begin{proof}
One may assume that $(X,T)$ is a unique ergodic system by Jewett-Krieger
Theorem \cite[Chapter 15.8]{Glasner}. Let $\sigma_d=T\times T^2 \times \ldots \times T^{d}$.

\medskip

\noindent {\bf Claim: }\ {\em There is some $X_0\in \X$ with $\mu(X_0)=1$ such that for each $x\in X_0$, one has that
\begin{equation*}
\frac {1}{N} \sum_{n=0}^{N-1} \sigma_d^n \d_{\bf x}\longrightarrow
\mu^{d},\ N\to \infty, \quad \text{weakly in $M(X^d)$},
\end{equation*}
where ${\bf x}=(x,x,\ldots,x)\in X^d$ and $\mu^{d}=\mu\times \ldots \times \mu.$
}

\medskip

\begin{proof}[Proof of Claim]

By Lemma \ref{lem-bougain}, for each $1\le i<j\le d$, there is some $X_{i,j}\in \X$ such that for all  $x\in X_{i,j}$
\begin{equation*}
  \frac{1}{N} \sum_{n=0}^{N-1} (T^i\times T^j)^n\d_{(x,x)}\longrightarrow
\mu\times \mu, \ N\to \infty, \quad \text{weakly in $M(X\times X)$.}
\end{equation*}
Let $\displaystyle X_0=\bigcap_{1\le i<j\le d} X_{i,j}$. For any $x\in X_0$ we will show that $\displaystyle\frac {1}{N} \sum_{n=0}^{N-1} \sigma_d^n \d_{\bf x}\longrightarrow \mu^{d},\ N\to \infty$, weakly in $M(X^d)$.

For this aim let $\lambda$ be any weak limit point of the sequence $\{\frac {1}{N} \sum_{n=0}^{N-1} \sigma_d^n \d_{\bf x}\}$ in $M(X^n)$.
To show the claim, it suffices to prove $\lambda=\mu^d$.

\medskip

First we show that $\lambda$ is a joining for $\{(X,T^i)\}_{i=1}^d$. For $j\in \{1,\ldots, d\}$ let
$\nu$ be the projection measure of $\lambda$ on $(X,T^j)$. Then $\nu$ is the weak limit point of
the sequence $\{\frac{1}{N}\sum_{n=0}^{N-1}T^{jn}\d_x\}$. By Lemma \ref{lem-unique}, $(X,T^j,\mu)$ is uniquely ergodic, and hence
\begin{equation*}
 \lim_{N\to \infty}\frac{1}{N}\sum_{n=0}^{N-1}T^{jn}\d_x=\mu, \quad \text{weakly in $M(X)$}.
\end{equation*}
In particular, $\nu=\mu$. Thus $\lambda$ is a joining for $\{(X,\mathcal{X},\mu, T^i)\}_{i=1}^d$.

Now we show that $\lambda$ is pairwise independent. For all $i\neq j \in \{1,2,\ldots,d\}$, the projection on $(X, T^i)\times (X, T^j)$ is a weak limit point of the sequence $\{\frac {1}{N} \sum_{n=0}^{N-1} (T^i\times T^j)^n \d_{(x,x)}\}$ in $M(X^2)$, denoted by $\eta$. Since $x\in X_0\subset X_{i,j}$, one has that
\begin{equation*}
  \frac{1}{N} \sum_{n=0}^{N-1} (T^i\times T^j)^n\d_{(x,x)}\longrightarrow
\mu\times \mu, \ N\to \infty, \quad \text{weakly in $M(X\times X)$.}
\end{equation*}
In particular, $\eta=\mu\times \mu$.

Since $(X,\mathcal{X},\mu,T)$ is PID m.p.t., so is $(X,\mathcal{X},\mu,T^i)$  for any $i\in \mathbb{N}$
by Lemma \ref{lem:PIDn is PID}.
Moreover, $\{(X,\mathcal{X},\mu,T^i)\}_{i=1}^d$ are JPID by Theorem \ref{thm:PIDs are JPID}. Thus $\lambda=\mu^d$. The proof of the Claim is completed.
\end{proof}

To conclude we have shown for all $x\in X_0$, and for all $g_1,\ldots, g_d\in C(X)$
\begin{equation}\label{s4}
\frac{1}{N} \sum_{n=0}^{N-1}
   g_1(T^nx)g_2(T^{2n}x)\ldots g_d(T^{dn}x)
   \longrightarrow  \int g_1d\mu\int g_2d\mu \ldots \int g_dd\mu
\end{equation}
as $N\to \infty$.

\medskip

Now we show that for all $f_1,\ldots, f_d\in L^\infty(\mu)$,
\begin{equation*}
    \frac 1 N\sum_{n=0}^{N-1}f_1(T^nx)f_2(T^{2n}x)\ldots f_d(T^{dn}x)
\end{equation*}
converges $\mu$ a.s.

Without loss of generality, we assume that for all $1\le j\le d$,
$\|f_j\|_\infty\le 1$. For any $\d>0$, choose continuous functions $g_j$ such that
$\|g_j\|_\infty\le 1$ and $\|f_j-g_j\|_{L^1}<\d / d $ for all $1\le j\le
d$.

Since $\|f_j\|_\infty\le 1, \|g_j\|_\infty\le 1$ and $\|f_j-g_j\|_{L^1}<\d / d $, by Lemma \ref{lem-product} we have
\begin{equation*}
\begin{split}
 \left | \frac{1}{N} \sum_{n=0}^{N-1}\prod_{j=1}^{d} f_j(T^{jn}x) -\frac{1}{N} \sum_{n=0}^{N-1}
   \prod_{j=1}^d g_j(T^{jn}x)\right |
&= \left | \frac{1}{N} \sum_{n=0}^{N-1}\Big[ \prod_{j=1}^{d} f_j(T^{jn}x) - \prod_{j=1}^{d}
g_j(T^{jn}x)\Big] \right |\\
&\le  \sum_{j=1}^d \Big[\frac {1}{N} \sum_{n=0}^{N-1}\Big | f_j(T^{jn}x)-g_j(T^{jn}x)\Big |\Big ],
\end{split}
\end{equation*}
and

\begin{equation}\label{s7}
\begin{split}
\left | \int_{X^d}\bigotimes_{j=1}^dg_j d\mu^{d}-
\int_{X^d}\bigotimes_{j=1}^df_jd\mu^{d} \right |
\le \sum_{j=1}^d\int_{X}|g_j-f_j|d\mu \le \d.
\end{split}
\end{equation}

Now by Birkhoff pointwise ergodic theorem we have that for all $1\le j\le d$
\begin{equation*}
    \frac {1}{N} \sum_{n=0}^{N-1}\Big | f_j(T^{jn}x)-g_j(T^{jn}x)\Big | \lra
\|f_j-g_j\|_{L^1}<\d/d, \quad N\to \infty.
\end{equation*}
for $\mu$ a.s.
Hence there is some $X^\d\in \X$ such that $\mu(X^\d)=1$ and for all $x\in X^\d$
\begin{equation}\label{s5}
\begin{split}
\left | \frac{1}{N} \sum_{n=0}^{N-1} \prod_{j=1}^{d} f_j(T^{jn}x) -\frac{1}{N} \sum_{n=0}^{N-1}
   \prod_{j=1}^d g_j(T^{jn}x)\right |
\le &  \sum_{j=1}^d \Big[\frac {1}{N} \sum_{n=0}^{N-1}\Big |
f_j(T^{jn}x)-g_j(T^{jn}x)\Big |\Big ]\\
\lra &  \sum_{j=1}^d \|f_j-g_j\|_{L^1} < \d,\  \quad N\to \infty.
\end{split}
\end{equation}

Now let $x\in X_0\cap X^\d$. By (\ref{s4}),
\begin{equation}\label{s6}
   \frac{1}{N} \sum_{n=0}^{N-1}
   \prod_{j=1}^d g_j(T^{jn}x)\to \int_{X^d}\bigotimes_{j=1}^d g_j\
   d \mu^{d}, N\to \infty.
\end{equation}

So combining (\ref{s7})-(\ref{s6}), we have for all $x\in X_0\cap X^\d$, when $N$ is large enough
\begin{equation*}
   \left |\frac{1}{N} \sum_{n=0}^{N-1}\prod_{j=1}^{d} f_j(T^{jn}x)-\int_{X^d}\bigotimes_{j=1}^df_jd\mu^{(d)}\right |<3\d.
\end{equation*}

Let $\displaystyle X'=X_0\cap \bigcap_{n=1}^\infty X^{\frac 1n}$. Then $\mu(X')=1$ and for all $x\in X'$,
\begin{equation*}
\lim_{N\to \infty} \frac{1}{N} \sum_{n=0}^{N-1}
   f_1(T^nx)f_2(T^{2n}x)\ldots f_d(T^{dn}x)
   = \int_{X} f_1 d\mu \ldots \int_X f_{d} d\mu.
\end{equation*}
The proof is completed.
\end{proof}

As applications of Theorem \ref{thm:PID implies a.e}, one has the following corollaries.

\begin{cor}
Let $(X,\X,\mu,T)$ be a finite-rank  mixing m.p.t. Then for all $d\in\N$ and all $f_1,\ldots,f_d\in L^\infty(X,\X,\mu)$,
\begin{equation*}
    \frac 1 N\sum_{n=0}^{N-1}f_1(T^nx)\ldots
f_d(T^{dn}x)\stackrel{a.s.}{\longrightarrow } \int f_1d\mu\int f_2d\mu \ldots \int f_dd\mu , \quad N\to \infty.
\end{equation*}
\end{cor}
\begin{proof} This follows from Theorem \ref{Ryzhikov-thm} and Theorem \ref{thm:PID implies a.e}.
\end{proof}

\begin{cor}
Let $(X,\X,\mu, T)$ be a simple\footnote{$(X,\X,\mu,T)$ is simple if the centralizer of the action $T$ is a group
and for every $k$ each ergodic $k$-joining of $X$ is a POOD (a product of off-diagonals). See \cite{DJR87} for more details.} m.p.t, then for all $d\in\N$ and all $f_1,\ldots,f_d\in L^\infty(X,\X,\mu)$, $\frac 1 N\sum_{n=0}^{N-1}f_1(T^nx)\ldots
f_d(T^{dn}x)$ converges a.s. Moreover, if $(X,T)$ is in addition  weakly mixing, then
\begin{equation*}
    \frac 1 N\sum_{n=0}^{N-1}f_1(T^nx)\ldots
f_d(T^{dn}x)\stackrel{a.s.}{\longrightarrow } \int f_1d\mu\int f_2d\mu \ldots \int f_dd\mu , \quad N\to \infty.
\end{equation*}
\end{cor}

\begin{proof}
Since each simple system is either a group system \cite[Theorem 2.3]{DJR87}, or a weakly mixing PID system
(the argument before \cite[Lemma 5.2]{DJR87}), the result follows.
\end{proof}

\begin{cor}[\cite{Assani1998}]
Let $(X,\X,\mu, T)$ be a weakly mixing m.p.t. such that the restriction of $T$ to its Pinsker algebra has spectral type singular w.r.t Lebesgue measure.
Then for all $d\in\N$ and all $f_1,\ldots,f_d\in L^\infty(X,\X,\mu)$,
\begin{equation*}
    \frac 1 N\sum_{n=0}^{N-1}f_1(T^nx)\ldots
f_d(T^{dn}x)\stackrel{a.s.}{\longrightarrow } \int f_1d\mu\int f_2d\mu \ldots \int f_dd\mu , \quad N\to \infty.
\end{equation*}
\end{cor}
\begin{proof} This follows from Derrien and Lesigne's theorem \cite{DL1996},  Theorem \ref{Host-thm} and Theorem \ref{thm:PID implies a.e}.
\end{proof}


\begin{cor}\label{cor-4.6}
Let $Y=(X\times_{\sigma}K/H,\mathcal{X},\mu,T)$ be a weakly mixing
isometric extension of a PID action and $f_{1},\ldots,f_{d}\in L^{\infty}(\mu)$,
then $\frac{1}{N}\sum_{n=0}^{N-1}f_{1}(T^{n}y)\cdots f_{d}(T^{dn}y)$
converge a.s.
\end{cor}
\begin{proof} This follows from Theorem \ref{thm:A-weakly-mixing isometric of PID is PID} and Theorem \ref{thm:PID implies a.e}.
\end{proof}

\begin{thm}
Let $(X,\mathcal{X},\mu,T)$ be a weakly mixing measurable
distal extension of a PID system and $f_{1},\ldots, f_{d}\in L^{\infty}(\mu)$,
then $\frac{1}{N}\sum_{n=1}^{N}f_{1}(T^{n}x)\cdots f_{d}(T^{dn}x)$
converge a.s.
\end{thm}

\begin{proof}
By the Furstenberg-Zimmer structure theorem $X$ is an $I$-extension
over its PID factor (\cite[Theorem 10.8]{Glasner}) . We now proceed as
in \cite[Subsection 6.3]{HSY} by Corollary~ \ref{cor-4.6}.
\end{proof}

\begin{cor}
In the following classes $(\ref{MEA-one})$ exists a.s.:
\begin{itemize}
\item Weakly mixing measurable distal extension of a transformation with spectral type singular w.r.t Lebesgue measure.
\item Weakly mixing measurable distal extension of a transformation with
the $R$-property.
\item Weakly mixing measurable distal extension of finite rank mixing.
\end{itemize}
\end{cor}

\begin{proof} This follows from Corollary \ref{cor-4.6},  Theorem \ref{Host-thm}, the discussion after Question \ref{prob: invertible zero-entropy weakly mixing is PID} and Theorem \ref{Ryzhikov-thm}.
\end{proof}

It seems the following has not been observed although this is straightforward
from the fact that a m.p.t. with spectral type singular w.r.t Lebesgue measure is generic
with respect to the weak topology as defined in \cite{Hal56}:

\begin{thm}
For a generic  m.p.t $(\ref{MEA-one})$ exists a.s.
\end{thm}

\begin{proof}
It is well known that the class of weakly mixing transformations is
generic \cite{Hal56}. A rigid transformation is always of singular
type (\cite[Theorem 5.11]{Glasner}.). Finally rigid transformations are
generic (See \cite[p. 86 Subsection 3]{KS67} or the proof of \cite[Theorem 3.1]{AS01}).
\end{proof}


\end{document}